\documentclass[11pt]{article}
\usepackage{fullpage}

\usepackage{latexsym}
\usepackage{amsmath}
\usepackage{amssymb}
\usepackage{amsthm}
\usepackage{hyperref}
\usepackage{cite}
\usepackage{graphicx}
\usepackage{color}
\usepackage{framed}
\usepackage{xspace}
\usepackage{subfigure}
\usepackage{cleveref}
\usepackage{enumitem}
\usepackage{nicefrac}
\usepackage{mathtools}

\usepackage{tikz}
\usetikzlibrary{fit, matrix}

\newtheorem{theorem}{Theorem}[section]

\newtheorem{theorem*}{Theorem}
\newtheorem{corollary}[theorem]{Corollary}
\newtheorem{lemma}[theorem]{Lemma}
\newtheorem{observation}[theorem]{Observation}
\newtheorem{proposition}[theorem]{Proposition}
\newtheorem{definition}[theorem]{Definition}

\newtheorem{fact}[theorem]{Fact}

\theoremstyle{definition}
\newtheorem{problem}{Problem}

\newtheorem{remark}[theorem]{Remark}

\renewcommand{\vec}[1]{\mathbf{#1}}

\newcommand{\GL}{\mathrm{GL}}
\newcommand{\PGL}{\mathrm{PGL}}

\newcommand{\Aut}{\mathrm{Aut}}

\newcommand{\conjc}{\mathrm{Conj}}
\newcommand{\fdim}{\mathrm{f}}
\newcommand{\F}{\mathbb{F}}

\newcommand{\Fix}{\mathrm{Fix}}
\newcommand{\Z}{\mathbb{Z}}
\newcommand{\Q}{\mathbb{Q}}

\newcommand{\N}{\mathbb{N}}
\newcommand{\M}{\mathrm{M}}

\newcommand{\id}{\mathrm{id}}

\newcommand{\tr}[1]{{#1}^{\mathrm{t}}}

\newcommand{\Spec}{\mathrm{Spec}}
\newcommand{\mul}{\mathrm{Mul}}

\newcommand{\Gr}{\mathrm{Gr}}
\newcommand{\twop}{{\mathfrak{B}_{p,2}}}
\newcommand{\higp}{{\mathfrak{H}_p}}
\newcommand{\higtwo}{{\mathfrak{H}_2}}

\newcommand{\cA}{\mathcal{A}}

\newcommand{\cG}{\mathcal{G}}

\newcommand{\linspan}{\mathrm{span}}

\newcommand{\tP}{\conjc(P)}
\newcommand{\tQ}{\conjc(Q)}

\renewcommand{\S}{\mathrm{S}}

\title{
On average orders of automorphism groups \\
of bilinear maps over finite fields
}

\author{
Markus Bl\"aser 
\thanks{Saarland University, Saarland Informatics Campus, Saarbr{\"u}cken, Germany (\texttt{mblaeser@cs.uni-saarland.de}).}
\and 
Yinan Li \thanks{School of Artificial Intelligence, Wuhan University, National Center for Applied Mathematics in Hubei, and Hubei Key Laboratory of Computational Science, Wuhan, China (\tt Yinan.Li@whu.edu.cn).}
\and 
Youming Qiao \thanks{
	Centre for Quantum Software and Information, University of Technology Sydney, Sydney, Australia  ({\tt jimmyqiao86@gmail.com}).}
\and Alexander Rogovskyy
\thanks{Saarland University, Saarland Informatics Campus, Saarbr{\"u}cken, Germany (\texttt{rogovskyy@cs.uni-saarland.de}).}
}

\date{\today}

\begin{document}

\maketitle

\begin{abstract}
	Let $\varphi:V\times V\to W$ be a bilinear map of finite vector spaces $V$ and $W$ over a finite field $\F_q$. We present asymptotic bounds on the number of isomorphism classes of bilinear maps under the natural action of $\GL(V)$ and $\GL(W)$, when $\dim(V)$ and $\dim(W)$ are linearly related.

	As motivations and applications of the results, we present almost tight upper bounds on the number of $p$-groups of Frattini class $2$ as first studied by Higman (\emph{Proc. Lond. Math. Soc.}, 1960). Such bounds lead to answers for some open questions by Blackburn, Neumann, and Venkataraman (\emph{Cambridge Tracts in Mathematics}, 2007). Further applications include sampling matrix spaces with the trivial automorphism group, and asymptotic bounds on the number of isomorphism classes of finite cube-zero commutative algebras.
%
%
\end{abstract}

\section{Introduction}

\subsection{Main results}



Let $q$ be a prime power, and let $\F_q$ be the finite field of order $q$. We consider bilinear maps $\varphi: V\times V\to W$, where $V\cong \F_q^n$ and $W\cong \F_q^m$.
A bilinear map $\varphi: V\times V\to W$ is \emph{symmetric}, if $\varphi(v_1,v_2)=\varphi(v_2,v_1)$ for any $v_1,v_2\in V$. It is \emph{alternating}, if $\varphi(v,v)=0$ for all $v\in V$. 

The general linear groups $\GL(V)$ and $\GL(W)$ naturally act on bilinear maps $\varphi: V\times V\to W$ as follows. For any $P\in\GL(V)$, $Q\in\GL(W)$, $(P,Q)\circ\varphi:V\times V\to W$ is defined by 
\[
[(P,Q)\circ\varphi](v_1,v_2)=Q^{-1}(\varphi(P(v_1),P(v_2))), 
\]
for $v_1, v_2\in V$.
We say two bilinear maps $\varphi_1,\varphi_2: V\times V\to W$ are \emph{isomorphic}, if there exist $P\in\GL(V)$ and $Q\in\GL(W)$ such that $(P,Q)\circ\varphi_1=\varphi_2$. It is clear that this action of $\GL(V)$ and $\GL(W)$ preserves symmetric and alternating properties of bilinear maps. 

For a bilinear map $\varphi:V\times V\to W$, let $\Aut(\varphi)=\{(P, Q)\in\GL(V)\times\GL(W)\mid (P, Q)\circ \varphi=\varphi\}$. Let $\id_V$ and $\id_W$ be the identity map in $\GL(V)$ and $\GL(W)$, respectively. We say that $\Aut(\varphi)$ is trivial, if $\Aut(\varphi)=\{(\alpha \id_V, \beta\id_W)\mid \alpha, \beta\in \F_q^{\times}, \alpha^2\beta=1\}$. 

Recall that $V\cong \F_q^n$ and $W\cong \F_q^m$. Let 
$$h(n, m, q):=\frac{\sum_{\varphi:V\times V\to W}|\Aut(\varphi)|}{|\{\varphi:V\times V\to W\}|}$$ 
be the average order of automorphism groups over all $\varphi:V\times V\to W$. Similarly, we can define average automorphism group orders for symmetric or alternating bilinear maps, denoted by $h_\S(n, m, q)$ and $h_\Lambda(n, m, q)$, respectively. 

The main result of this paper is the following.
\begin{theorem}\label{thm: bilinear maps}
	For $C, R\in \Q$ with $C>0$ and $n\in \N$, let $m=m(n)=\lceil C\cdot n+R\rceil$. 
For large enough $n\in \N$, $h(n, m, q)$, $h_\S(n, m, q)$, and $h_\Lambda(n, m, q)$ are upper bounded by  $q-1+\frac{1}{q^{\Omega(n^2)}}$. 
\end{theorem}

By Burnside's lemma (also known as Cauchy-Frobenius lemma, see~\cite[Page 42]{robinson96} and~\cite{Eick_O’Brien_1999}), Theorem~\ref{thm: bilinear maps} implies that the number of isomorphism classes of bilinear maps $\varphi: V\times V\to W$ is upper bounded by  $(q-1+\frac{1}{q^{\Omega(n^2)}})\cdot \frac{q^{n^2m}}{|\GL(n, q)|\cdot |\GL(m, q)|}$, and similarly for symmetric and alternating bilinear maps. 



Theorem~\ref{thm: bilinear maps} is motivated by, and therefore has applications to, finite group enumeration, matrix space random sampling, and finite commutative algebra enumeration. In the following, we explain these applications.

\subsection{Application 1: enumerating \texorpdfstring{$p$}{p}-groups of Frattini class \texorpdfstring{$2$}{2}}

Groups in this section are finite. 

Let $\twop$ be the set of $p$-groups of class $2$ and exponent $p$ (up to isomorphism). That is, a finite group $G$ is in $\twop$ if the commutator subgroup $[G, G]$ is contained in the centre $Z(G)$, and for any $g\in G$, $g^p=\id$. As every group of exponent $2$ is abelian, for $\twop$ the more interesting setting is when $p>2$. 

Let $\higp$ be the set of $p$-groups of Frattini class $2$ (up to isomorphism). That is, a finite group $G$ is in $\higp$ if its Frattini subgroup is central and elementary abelian. Or equivalently, $G$ satisfies that for any $g, h, k\in G$, $g^{p^2}=\id$, $[g, h]^p=\id$, and $[g, [h, k]]=\id$. 

Let $f(n)$ denote the number of (isomorphism classes of) groups of order $n$. Let $f_\twop(p^{\ell})$ (resp.\ $f_\higp(p^\ell)$) be the number of groups in $\twop$ (resp.\ $\higp$) of order $p^{\ell}$.

In 1960, Higman presented a pair of influential papers on enumerating $p$-groups \cite{Hig60,Hig60b}. To give a lower bound of $f(p^\ell)$, Higman showed that  $f_\higp(p^\ell)\geq p^{\frac{2}{27}(\ell^3-6\ell^2)}$~\cite{Hig60}. In \cite{Sim65}, Sims proved that $f(p^\ell)\leq p^{\frac{2}{27}\ell^3+O(\ell^{8/3})}$, with the error term improved from $O(\ell^{8/3})$ to $O(\ell^{5/2})$ by Newman and Seeley \cite{NS_enumeration}. Starting from Neumann \cite{Neu69}, several works studied $f(n)$ for general $n$ \cite{MN87,Holt89}, leading to the celebrated work of Pyber~\cite{Pyb93}, who proved that $f(n)\leq n^{\frac{2}{27}\mu(n)^2+O(\mu(n)^{5/3})}$, where $\mu(n)$ denotes the largest exponent in the prime decomposition of $n$. These works and more are presented in the monograph by Blackburn, Neumann, and Venkataraman \cite{BNV07}.


From the above brief account, we see that $p$-groups of Frattini class $2$, and $p$-groups of class $2$ and exponent $p$, are important group classes in group enumeration. Indeed, the constant $\frac{2}{27}$ in the above results can be traced back to Higman's bounds on $f_\higp$. From the group variety viewpoint, $\twop$ and $\higp$ serve as the minimal non-abelian $p$-group varieties for $p>2$ and $p=2$, respectively, in the sense that any non-abelian $p$-group variety contains them \cite[Theorem 19.1]{BNV07}. The abundance of groups in $\twop$ and $\higp$ hints at the difficulty of devising isomorphism testing algorithms for these groups. Indeed, they are generally viewed as bottlenecks for finite group isomorphism \cite{OBr93,CH03,BCGQ11}, with notable progress only in the past few years \cite{LQ17,BLQW20,Sun23,Frattini}.

The following result is due to Higman \cite{Hig60}, as worked out in detail in \cite{Bla92,BNV07}.

\begin{theorem}[{\cite{Hig60}, see \cite[Chapter 2.1]{Bla92} and \cite[Theorems 19.2 and 19.3]{BNV07}}]\label{thm:prior}
	Let $f_\twop$ and $f_\higp$ be as above. Then we have 
	$$
	p^{\frac{2}{27}\ell^3-\frac{2}{3}\ell^2}\leq f_{\twop}(p^\ell)\leq \ell\cdot 
	p^{\frac{2}{27}\ell^3-\frac{2}{9}\ell^2+\frac{49}{72}\ell}~\text{for}~p\geq 3,
	$$
	and 
	$$
	p^{\frac{2}{27}\ell^3-\frac{4}{9}\ell^2}\leq f_{\higp}(p^\ell)\leq \ell\cdot 
	p^{\frac{2}{27}\ell^3+\frac{11}{24}\ell}.
	$$
\end{theorem}
Since the coefficient of the cubic term on the exponent is tight 
(i.e. $\frac{2}{27}$), a natural question is to examine the coefficients of the 
quadratic term. This was posed as open problems in \cite{BNV07}.
\begin{problem}
	\label{prob:main}
	\begin{enumerate}
		\item[(1)] (Part of \cite[Question 22.8]{BNV07}) Suppose $$f_\twop(p^\ell)=p^{\frac{2}{27}\ell^3+E_\twop(\ell)}.$$ Does $\lim_{\ell\to\infty}E_\twop(\ell)/\ell^2$ exist? If so, what is this limit? 
		
		The same question was asked for $\higp$ instead of $\twop$ in the above. 
		\item[(2)] (\cite[Question 22.11]{BNV07}) Is it true that 
		$$f_\higp(p^\ell)/f_\twop(p^\ell)\geq p^{\frac{2}{9}\ell^2-O(\ell)}?$$
		Could it perhaps even be true that  
		$f_\twop(p^\ell)=p^{\frac{2}{27}\ell^3-\frac{2}{9}\ell^2+O(\ell)},$
		and
		$f_\higp(p^\ell)	=p^{\frac{2}{27}\ell^3+O(\ell)}?$
	\end{enumerate}
\end{problem}

In this paper, we answer Problem~\ref{prob:main} in the following theorem; its proof is in Section~\ref{subsec:group-enum}.
\begin{theorem}\label{thm:group-enum}
	Let $f_\twop$ and $f_\higp$ be as above. Then we have	$$p^{\frac{2}{27}\ell^3-\frac{2}{3}\ell^2}\leq f_\twop(p^\ell)\leq p^{\frac{2}{27}\ell^3-\frac{2}{3}\ell^2+O(\ell)}~\text{for}~p\geq 3,$$
	and 
	$$p^{\frac{2}{27}\ell^3-\frac{4}{9}\ell^2}\leq f_\higp(p^\ell)\leq p^{\frac{2}{27}\ell^3-\frac{4}{9}\ell^2+O(\ell)}.$$
\end{theorem}

Theorem~\ref{thm:group-enum} answers Problem~\ref{prob:main} (1) as 
$$\lim_{\ell\to\infty}\frac{E_\twop(\ell)}{\ell^2}=-\frac{2}{3}~\text{for}~p\geq 3,~\lim_{\ell\to\infty}\frac{E_\higp(\ell)}{\ell^2}=-\frac{4}{9}$$
and answers Problem~\ref{prob:main} (2) as 
$$p^{\frac{2}{9}\ell^2-\Omega(\ell)}\leq \frac{f_\higp(p^\ell)}{f_\twop(p^\ell)}\leq p^{\frac{2}{9}\ell^2+O(\ell)}.$$

\subsection{Application 2: random matrix spaces with trivial symmetry}

A classic question in random graph theory is the enumeration of unlabelled graphs with $n$ vertices and $m$ edges, that is the number of orbits of $n$-vertex $m$-edge graphs under the action of the symmetric group of degree $n$. This was studied by P\'olya (see \cite{FU56}) and Oberschelp \cite{Ober67}. In \cite{Wri71}, Wright proved the following major result. 
\begin{theorem}[{\cite{Wri71}}]\label{thm:wright}
	For $n, m\in \N$, let $g(n, m)$ be the number of unlabelled graphs with $n$ vertices and $m$ edges. Then 
	$$g(n, m)\sim \frac{\binom{\binom{n}{2}}{m}}{n!}$$
	as $n\to \infty$ if and only if 
	$$\frac{1}{2}n\log n+\omega(n)n\leq m\leq \binom{n}{2}-\frac{1}{2}n\log n-\omega(n)n,$$ where $\omega(n)$ is a function tending to $\infty$. 
\end{theorem}
This result is useful for translating results from random labelled graphs to random unlabelled graphs \cite[Chapter 9]{Bol01}. It also implies that for $m$ in that range, most graphs are asymmetric, i.e. have the trivial automorphism group, a classical topic in random graph theory \cite{ER63,Bol01}.



Some works indicate that alternating matrix spaces may be studied as a linear algebraic analogue of graphs \cite{Lov89,BCGQS21,LQ17,LQ20,Qia21,Qia23,LQWWZ,Qia24}. In particular, over finite fields, automorphism groups of random alternating matrix spaces were studied in \cite{LQ17,BLQW20}, and threshold phenomena for monotone properties of random matrix spaces were shown in \cite{Ros20}.

Theorem~\ref{thm: bilinear maps} leads to the following enumeration bounds of isomorphism classes of matrix spaces. We shall focus on $\Lambda(n, q)$, the linear space of $n\times n$ alternating matrices over $\F_q$, though analoguous statements are also valid for linear spaces of general or symmetric matrices. Let $\Gr(U, m)$ be the Grassmannian of $m$-dimensional subspaces of a vector space $U$. The congruence action of $\GL(n, q)$ on $\Lambda(n, q)$ is defined as $P\in\GL(n, q)$ sending $A\in\Lambda(n, q)$ to $\tr{P}AP$. Naturally it induces an action on $\Gr(\Lambda(n, q), m)$, that is, for $\cA\in \Gr(\Lambda(n, q), m)$, $P\in\GL(n, q)$ sends $\cA$ to $\tr{P}\cA P:=\{\tr{P}AP\mid A\in\cA\}$. 

Let $g(\Lambda(n,q), m)$ be the number of orbits of the congruence action of $\GL(n, q)$ on $\Gr(\Lambda(n, q))$. The following is a relatively straightforward consequence of Theorem~\ref{thm: bilinear maps}; its proof is in Section~\ref{subsec:cor-space}. 


\begin{theorem}\label{thm:space}
For $C, R\in \Q$ with $C>0$ and $n\in \N$, let $m=m(n)=\lceil C\cdot n+R\rceil$. 
	Let $K=q^{m^2}/|\GL(m,q)|$. Then 
	$$
	g(\Lambda(n,q), m) \sim K\cdot \frac{\binom{\binom{n}{2}}{m}_q}{|\PGL(n, q)|}
	$$
	as $n\to \infty$. Furthermore, $g(\Lambda(n,q), m)\sim \frac{\binom{\binom{n}{2}}{m}_q}{|\PGL(n, q)|}$ as $n, q\to \infty$.
\end{theorem}
Note that compared to Theorem~\ref{thm:wright}, the range $m=\Theta(n)$ in Corollary~\ref{thm:space} is restrictive, and it is an interesting open problem to relax it. Meanwhile, we consider the action of the projective linear group instead of the general linear group, since we may omit the scalar action as they trivially induce automorphisms between matrix spaces.


Theorem~\ref{thm:space} has implications to sampling matrix spaces with trivial symmetry as follows. Define the automorphism group of $\cA\in\Gr(\Lambda(n, q),m)$ as $\Aut(\cA):=\{P\in\GL(n, q)\mid \tr{P}\cA P=\cA\}$. We say that $\Aut(\cA)$ is trivial if it just consists of non-zero scalar matrices. 

In \cite{LQ17,BLQW20}, the problem of estimating the probability of sampling matrix spaces with the trivial automorphism group was studied in connection with testing isomorphism of $p$-groups of class $2$ and exponent $p$. This question is also of interest in cryptography. Some digital signature schemes \cite{MEDSspecs,ALTEQspecs} are based on the assumed hardness of (variations of) testing congruence of matrix spaces. For these schemes, trivial automorphism groups are useful for showing security in the so-called quantum random oracle model \cite{BCDJNPQST24}.

In \cite{BLQW20}, it was shown that when $\Omega(1)\leq m\leq \binom{n}{2}-O(1)$, for all but $1/q^{\Omega(n^2)}$ fraction of $\cA\in\Gr(\Lambda(n, q), m)$, the order of $\Aut(\cA)$ is upper bounded by $q^{O(n+m)}$. However, that bound should be far from tight; it is expected that a random matrix space should have the trivial automorphism group as long as $n$, $m$, and $q$ are reasonably large constants. 

By Theorem~\ref{thm:space}, we can obtain the following corollary, and its proof is in Section \ref{subsec:cor-space}. 

\begin{corollary}\label{cor:fraction}
For $C, R\in \Q$ with $C>0$ and $n\in \N$, let $m=m(n)=\lceil C\cdot n+R\rceil$. Let $K=q^{m^2}/|\GL(m,q)|$.
    When $n$ is large enough, all but $\frac{1}{q^{\Omega(n^2)}}$ fraction of $m$-dimensional subspaces $\Lambda(n, q)$ have their automorphism group order upper bounded by $K\cdot (q-1)$.
\end{corollary}



\subsection{Application 3: enumerating cube-zero finite commutative algebras}

Closely related to enumerating finite groups is enumerating finite rings \cite{KP70}. This problem is useful in bounding dimensions of varieties of Lie or associative algebras \cite{Ner88}. Recently, Poonen connected this to estimating dimensions of the moduli space of rank-$n$ commutative algebras with an ordered basis, with applications to dimension bounds for the Hilbert scheme of $n$ points in $\mathbb{A}^d$ when $d$ is linearly related to $n$ \cite{Poonen}. In \cite{BM22}, a problem in the proof of \cite{KP70} (which also affects \cite{Poonen}) was corrected, with an improved error term in the upper bounds in \cite{KP70,Poonen}.


Let $A$ be a commutative algebra over $\F_q^n$. We say that $A$ is cube-zero if $A^3=0$. When enumerating finite algebras, cube-zero commutative algebras over finite fields play the role of $p$-groups of class $2$ and exponent $p$ in enumerating finite groups. The variety of cube-zero commutative algebras has been studied in algebraic geometry, as seen in the work of Shafarevich in the 1990's \cite{Sha90} and more recently by Poonen \cite{Poonen}. 

In the literature (see \cite[Lemma 9.1]{Poonen} and \cite[Theorem 4.1]{BM22}), the number of (isomorphism classes of) cube-zero commutative algebras over $\F_q^n$ was estimated as $q^{\frac{2}{27}n^3+\Theta(n^2)}$. A finer analysis shows that the coefficient of the quadratic term on the exponent is bounded between $-\frac{4}{9}$ and $0$. As can be seen from \cite[Theorem 2.1]{KP70} and \cite[Lemma 9.1]{Poonen}, isomorphism classes of cube-zero commutative algebras over $\F_q$ correspond to congruence classes of symmetric matrix spaces. Therefore, by the statement for symmetric bilinear maps in Theorem~\ref{thm: bilinear maps}, through an analysis similar to what is done for groups (Section~\ref{sec:aux}), we can obtain that the number of cube-zero commutative algebras is $q^{\frac{2}{27}n^3-\frac{4}{9}n^2+\Theta(n)}$.

%
%
%

\paragraph{Organisation of the paper.} In Section~\ref{sec:prep}, we collect basic notions and present certain facts about conjugacy classes of $\GL(n, q)$ that are useful to us. In Section~\ref{sec:proof} we prove Theorem~\ref{thm: tensors} and Corollaries~\ref{thm:space}. In Section~\ref{sec:aux} we prove Theorem~\ref{thm:group-enum}.


\section{Preliminary}\label{sec:prep}

\paragraph{Notations.} For $n\in\N$, $[n]:=\{1, 2, \dots, n\}$. For a vector space $V$ and $S\subseteq V$, the span of $S$ in $V$ is denoted by $\linspan(S)$. Let $\GL(n,q)$ be the general linear group of order $n$ over $\F_q$ and $\PGL(n, q)$ its projective version.

\subsection{Bilinear maps and matrix spaces}

Let $V\cong \F_q^n$ and $W\cong \F_q^m$, where $q$ is a prime power. 
The tensor product space $V\otimes V\otimes W^*$ is naturally identified as the space of bilinear maps $V\times V\to W$. Enumerating isometric bilinear maps over finite fields is equivalent to the number of orbits of $V\otimes V\otimes W^*$ under the natural action of $\GL(V)\times\GL(W)$. For this purpose, we may as well consider $V\otimes V\otimes W$ instead of $V\otimes V\otimes W^*$. 
So in the following, we shall not distinguish $W^*$ from $W$. 

Similarly, we use $\odot$ to denote the symmetric product, and $\wedge$ the wedge product. We are also interested in the orbits of $V\odot V\otimes W$ and $V\wedge V\otimes W$ under the natural action of $\GL(V)\times \GL(W)$. For convenience, we shall let  $\Box$ to denote one of the products in $\{\otimes,\odot,\wedge\}$.
Let $g(V\Box V\otimes W)$ be the number of orbits of $V\Box V\otimes W$ under the natural action of $\GL(V)\times\GL(W)$. 

Using the above notations, Theorem~\ref{thm: bilinear maps} can be reformulated as follows, thanks to Burnside's lemma which connects average automorphism group orders and orbit numbers.
\begin{theorem}\label{thm: tensors}
For $C, R\in \Q$ with $C>0$ and $n\in \N$, let $m=m(n)=\lceil C\cdot n+R\rceil$. 
	Let $V\cong \F_q^n$ and $W\cong \F_q^m$. 
    Then for large enough $n$, we have
	$$g(V\Box V\otimes W)\leq (q-1+\frac{1}{q^{\Omega(n^2)}})\cdot \frac{q^{\dim(V\Box V\otimes W)}}{|\GL(V)|\cdot|\GL(W)|}.$$
\end{theorem}


Bilinear maps naturally give rise to matrix spaces as follows. For a bilinear map $\varphi\in V\Box V\otimes W^*$, for $w\in W$, let $\varphi\circ w\in V\Box V$ be the evaluation of $\varphi$ on $w$. Then $\cA_\varphi = \{\varphi\circ w\mid w\in W\}$ is a subspace of $V\Box V$. We say that two subspaces of $V\Box V$ are congruent, if they are in the same orbit under the natural action of $\GL(V)$. It is easy to verify that $\varphi$ and $\psi$ are isomorphic if and only if $\cA_\varphi$ and $\cA_\psi$ are isometric. 

\subsection{Conjugacy classes of \texorpdfstring{$\GL(n, q)$}{glnq}}
We recall some results about the conjugacy classes of $\GL(n, q)$. For $P\in \GL(n, q)$, the conjugacy class of $\GL(n, q)$ containing $P$ is $\conjc(P):=\{TPT^{-1}\mid T\in\GL(n, q)\}$.
We first note the following result about the total number of conjugacy classes.
\begin{lemma}[See e.g. \cite{Mac81}]\label{lem:total_num_conj_classes}
	The number of conjugacy classes in $\GL(n, q)$ is a monic polynomial in $q$ of degree $n$.
\end{lemma}
We then estimate the sizes of the conjugacy classes of $\GL(n, q)$. A useful tool is the rational normal form of matrices (cf.~\cite[Sec. 6.7]{herstein1991topics}). 
\begin{definition}[Rational normal forms]
	Let $P \in \GL(n, q)$ and $X$ be a formal variable. Then there are unique irreducible monic polynomials $f_1,\dots,f_k \in \F_q[X]$ and, for each $f_i$, unique exponents $e_{i,1},\dots,e_{i,r_i}$, such that $P$ is conjugate to a block-diagonal matrix where each diagonal-block is the companion matrix of $f_i^{e_{i,j}}$. This representation is unique (up to permutations of the blocks).
\end{definition}

Rational normal forms are representatives for the conjugacy classes. That is, each conjugacy class in $\GL(n, q)$ is uniquely classified by
\begin{itemize}
	\item the irreducible monic polynomials $f_1, \dots, f_k\in \F_q[X]$ and
	\item for each $f_i$, the exponents $e_{i,1},\dots,e_{i,r_i}\in \Z^+$. These exponents yield a partition $\lambda_i=(e_{i,1},\dots,e_{i,r_i})$.
\end{itemize}
The polynomials $f_i^{e_{i,j}}$ where $i\in[k]$ and $j\in[r_i]$ are called the \emph{elementary divisors} of $P$.
\begin{remark}\label{rem:size_vector_space_partition}
	The subspace associated with the companion matrix $f_i^{e_{i,j}}$ is of dimension $\deg(f_i)\cdot e_{i,j}$. Thus, we have
	\begin{equation*}
		n = \sum_{i\in[k]} \sum_{j\in[r_i]} \deg(f_i)\cdot e_{i,j} = \sum_{i\in[k]} \deg(f_i) |\lambda_i|
	\end{equation*}
	where $|\lambda_i|$ is the sum of the parts in the partition $\lambda_i$.
\end{remark}
To exhibit explicit formulas for the size of conjugacy classes, we need the following notions.
\begin{definition}[$q$-Pochhammer symbol]
	The $q$-Pochhammer symbol is 
	$$(u; q)_k := \prod_{i=1}^k \left( 1 - uq^i\right).$$
\end{definition}
Using the $q$-Pochhammer symbol, we can write $|\GL(n, q)| = q^{n^2}\cdot (1;q^{-1})_n$.

\begin{definition}[Dual partition]
	Let $\lambda=(a_1, \dots, a_k)$ be a partition of $n\in \N$. Let
	\begin{equation*}
		m(\lambda, e) := \#\{a_i \in \lambda \mid a_i = e\}
	\end{equation*}
	for the number of parts in $\lambda$ that are equal to $e$. The \emph{dual partition} $\lambda'$ of $\lambda$ is then
	\begin{equation*}
		\lambda' := \left(\sum_{j\geq 1} m(\lambda, j),\;\; \sum_{j\geq 2} m(\lambda, j), \;\;\dots \right)
	\end{equation*}
\end{definition}

Visually, a partition and its dual can be represented by a \emph{Young diagram}: We arrange squares from left to right and top down, with each row representing one part, and the parts are sorted by decreasing size. The dual can be then seen by transposing the diagram (or, alternatively, considering the columns as parts instead of the rows). See Figure \ref{fig:young} for an example.

\begin{figure}[ht!]
	\centering
	\begin{tikzpicture}[x=0.5cm,y=0.5cm, baseline=(current bounding box.north)]
		\draw (0,0) rectangle node{} (1,1);
		\draw (1,0) rectangle node{} (2,1);
		\draw (0,-1) rectangle node{} (1,0);
		\draw (1,-1) rectangle node{} (2,0);
		\draw (0,-2) rectangle node{} (1,-1);
		\draw (0,-3) rectangle node{} (1,-2);
	\end{tikzpicture}
	\hspace{2cm}
	\begin{tikzpicture}[x={(0cm, -0.5cm)},y={(-0.5cm, 0cm)}, baseline=(current bounding box.north)]
		\draw (0,0) rectangle node{} (1,1);
		\draw (1,0) rectangle node{} (2,1);
		\draw (0,-1) rectangle node{} (1,0);
		\draw (1,-1) rectangle node{} (2,0);
		\draw (0,-2) rectangle node{} (1,-1);
		\draw (0,-3) rectangle node{} (1,-2);
	\end{tikzpicture}
	\caption{A Young diagram for the partition $(2, 2, 1, 1)$ (on the left) and its dual partition $(4,2)$ (on the right), which is obtained by transposing the diagram.}
	\label{fig:young}
\end{figure}
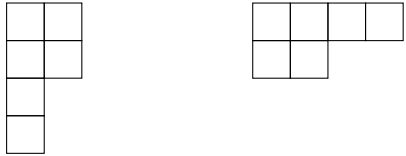
We are now ready to state formulas for the sizes of conjugacy classes.
\begin{lemma}
	Let $P\in\GL(n, q)$. Suppose $f_i^{e_{i,j}}$, $i\in[k]$ and $j\in[r_i]$, are the elementary divisors of $P$, and $\lambda_1,\dots,\lambda_k$ are the corresponding partitions. Then
	\begin{equation}
		\label{eq:conjclasssize}
		|\conjc(P)| = \frac{|\GL(n, q)|}{
			\prod_{i\in[k]} \big(q^{\deg(f_i)\cdot \sum_{e\in \lambda'_i} e^2} \cdot
			\prod_{e\in \lambda_i} (1;q^{-\deg(f_i)})_{m(\lambda_i, e)}\big)
		}
	\end{equation}
\end{lemma}
\begin{proof}
	From~\cite[pp. 181]{Mac98} and~\cite[pp. 55]{Ful02}, the number of matrices $T\in\GL(n, q)$ that commute with $P$ is 
	$$
	\prod_{i\in[k]} \big(q^{\deg(f_i)\cdot \sum_{e\in \lambda'_i} e^2} \cdot \prod_{e\in \lambda_i} (1;q^{-\deg(f_i)})_{m(\lambda_i, e)}\big).
	$$
	The formula then follows by the orbit-stabilizer theorem.
\end{proof}

The term in the denominator of Equation~\eqref{eq:conjclasssize} appears in the definition of the \emph{cycle index} generating function of $\GL(n, q)$~\cite[pp. 55]{Ful02}. The cycle index is originally a tool used to enumerate and count permutations. Kung \cite{kung1981cycle} and Stong \cite{stong1988some} derived a variation of the cycle index for linear groups, which was subsequently used to compute some enumerative formulas like the one for the size of a certain conjugacy class. 

We compute a more convenient estimate for the exponent of $|\conjc(P)|$ of a general $P\in\GL(n,q)$.


\begin{lemma}
	\label{lem:bound_size_conjclass}
	For any $P\in\GL(n,q)$,
	\begin{equation}\label{eq:simplified}
		\log_q |\conjc(P)| \leq n^2 + (1-\log_q(q-1))n - \sum_{i\in[k]} \deg(f_i) \sum_{e\in \lambda'_i} e^2.
	\end{equation}
\end{lemma}
\begin{proof}
	We first bound the term
	\begin{equation*}
		T_{n,q} := \frac{
			(1;q^{-1})_n
		}{
			\prod_{i\in[k]}\prod_{e\in \lambda_i}
			(1; q^{-\deg(f_i)})_{m(\lambda_i, e)}
		}
	\end{equation*}
	which occurs in Equation~\eqref{eq:conjclasssize}. (Recall that $|\GL(n, q)| = q^{n^2}\cdot (1;q^{-1})_n$.) We will focus on the denominator of this term. Note that
	\begin{equation*}
		\left(1 - \frac{1}{q^i}\right) \geq \left(1 - \frac{1}{q}\right) \quad \text{for all } i \geq 1.
	\end{equation*}
	Thus, we have
	\begin{align*}
		\prod_{i\in[k]}\prod_{e\in \lambda_i}
		(1; q^{-\deg(f_i)})_{m(\lambda_i, e)} &\geq 
		\prod_{i\in[k]}\prod_{e\in \lambda_i}
		(1; q^{-1})_{m(\lambda_i, e)}
		\\
		&\geq
		\prod_{i\in[k]}\prod_{e\in \lambda_i}
		\left(
		1 - \frac{1}{q}
		\right)^{m(\lambda_i, e)}
		=
		\left(
		1 - \frac{1}{q}
		\right)^{\sum_{i\in[k]}\sum_{e\in \lambda_i} m(\lambda_i, e)}.
	\end{align*}
	Furthermore, we can bound the sum in the exponent using Remark \ref{rem:size_vector_space_partition} as
	\begin{equation*}
		\sum_{i\in[k]} \sum_{e\in\lambda_i} m(\lambda_i, e) \leq \sum_{i\in[k]} \sum_{e\in\lambda_i} e \cdot m(\lambda_i, e) \leq \sum_{i\in[k]} |\lambda_i| \leq \sum_{i\in[k]} \deg(f_i) |\lambda_i|  = n.
	\end{equation*}
	Since $1-\nicefrac{1}{q} < 1$, we have
	\begin{align*}
		\prod_{i\in[k]}\prod_{e\in \lambda_i}
		(1; q^{-\deg(f_i)})_{m(\lambda_i, e)} &\geq 
		\left(
		1 - \frac{1}{q}
		\right)^{\sum_{i}\sum_{e\in \lambda_i} m(\lambda_i, e)}
		\geq
		\left(
		1 - \frac{1}{q}
		\right)^n.
	\end{align*}
	Returning to $T_{n,q}$, we have the following bound
	\begin{align*}
		T_{n,q} \leq \frac{
			(1;q^{-1})_n
		}{
			\left(
			1 - \frac{1}{q}
			\right)^n
		}
		&\leq
		\frac{
			\left(
			1 - \frac{1}{q^n}
			\right)^n
		}{
			\left(1 - \frac{1}{q}\right)^n
		}
		=
		\left(
		\frac{
			(q^n-1)\cdot q
		}{
			q^n\cdot (q-1)
		}
		\right)^n \leq \left(
		\frac{q}{q-1}
		\right)^n .
	\end{align*}
	We can bound Equation~\eqref{eq:conjclasssize} as
	\begin{equation*}
		|\conjc(P)| \leq \frac{
			q^{n^2} \cdot \left(
			\frac{q}{q-1}
			\right)^n
		}{
			\prod_{i\in[k]} q^{\deg(f_i)\cdot \sum_{e\in \lambda'_i} e^2}
		}
		=
		\frac{
			q^{n^2} \cdot \left(
			\frac{q}{q-1}
			\right)^n
		}{
			q^{\sum_{i\in[k]} \deg(f_i)\cdot \sum_{e\in \lambda'_i} e^2}
		}.
	\end{equation*}
	The statement follows by taking the $q$-logarithm on both sides.
\end{proof}

As a consequence of Equation~\eqref{eq:conjclasssize} and its simplification in Lemma~\ref{lem:bound_size_conjclass}, we have the following bound for $\conjc(P)$ when $P$ is ``almost scalar''.
\begin{lemma}
	\label{lem:almost_scalar_conj_size}
	Let $k \geq 1$ be a constant independent of $n$ and $q$. Suppose $P\in\GL(n, q)$ is  \emph{almost scalar}, that is, containing an eigenspace of dimension at least $n-k$ for an eigenvalue $\alpha\in\F_q$. Then
	\begin{equation*}
		|\conjc(P)| \leq q^{O(n)}.
	\end{equation*}
\end{lemma}
\begin{proof}
	One of the elementary divisors for $P$ is $X-\alpha$ with degree $1$ and its partition $\lambda$ consists at least $n-k$ ones. Thus, the corresponding dual partition has one single part with value at least $n-k$. Thus, 
	\[
	\sum_{i\in[k]} \deg(f_i) \sum_{e\in \lambda'_i} e^2\geq (n-k)^2
	\]
	and we get
	\begin{equation*}
		\log_q |\conjc(P)| \leq n^2 + (1-\log_q(q-1))n - (n-k)^2 = (2k+1-\log_q(q-1))n-k^2=O(n),
	\end{equation*}
	where we note that $(1-\log_q(q-1))$ is at most $1$.
\end{proof}

\section{Average order of automorphism groups of bilinear maps}\label{sec:proof}

\subsection{Proof of Theorem~\ref{thm: tensors}}\label{subsec:thm-tensors}

We prove Theorem~\ref{thm: tensors} (and therefore Theorem~\ref{thm: bilinear maps}) via the following strategy.

\paragraph{The setting.} 
Let $V$ and $W$ be vector spaces over $\F_q$, with $\dim(V)=n$, $\dim(W)=m$, and $m=\lceil Cn+R\rceil$ for constants $C$ and $R$. 

Note that $Cn+R\leq m\leq Cn+R+1$, and this difference of $1$ would not affect our arguments in the following. Therefore, for convenience, we shall work with $m=Cn+R$. The reader can also think of this as we restrict ourselves to considering those $n$ such that $Cn+R$ is an integer. For example, when $C=1/3$ and $R=2/3$, we shall study only those $n$ such that $n\mod 3=1$. 

\paragraph{Turn to study fixed points.} 
For $(P,Q)\in\GL(V)\times\GL(W)$, let 
$$\Fix((P, Q), V\Box V\otimes W):=\{t\in V\Box V\otimes W\mid (P, Q)\circ t=t\}.$$

By Burnside's lemma, 
$$g(V\Box V\otimes W)=\frac{\sum_{(P,Q)\in\GL(V)\times\GL(W)}|\Fix((P,Q), V\Box V\otimes W)|}{|\GL(V)|\cdot|\GL(W)|}.$$ 

By abbreviating $\Fix((P,Q), V\Box V\otimes W)$ as $\Fix(P,Q)$ and letting $d=\dim(V\Box V)$, the question becomes to prove that
\begin{equation}\label{eq:fix_sum}
	\sum_{(P,Q)\in\GL(V)\times\GL(W)}|\Fix(P,Q)|\cdot q^{-d\cdot m}\leq q-1+\frac{1}{q^{\Omega(n^2)}}.
\end{equation}

\paragraph{Partition according to conjugacy classes.} Let $\cG_V$ be the set of conjugacy classes of $\GL(V)$, and $\cG_W$ the set of conjugacy classes of $\GL(W)$. For $P\in \GL(V)$ (resp.\ $Q\in\GL(W))$, $\tP\in\cG_V$ (resp.\ $\tQ\in\cG_W$) is the conjugacy class containing $P$ (resp.\ $Q$). Suppose $P$ and $P'$ in $\GL(V)$ are conjugate, and $Q$ and $Q'$ in $\GL(W)$ are conjugate. Then it is clear that $|\Fix(P,Q)|=|\Fix(P', Q')|$. 

For $(P,Q)\in \GL(V)\times\GL(W)$, let $\mu_{P,Q}$ be the geometric multiplicity of the eigenvalue $1$ of the action of $(P,Q)$ on $V\Box V\otimes W$. Let 
$\nu_{P,Q}=d\cdot m-\mu_{P,Q}$. Then by grouping matrices according to conjugacy classes, we have 
$$\sum_{(P,Q)\in\GL(V)\times\GL(W)}|\Fix(P,Q)|\cdot q^{-d\cdot m}=\sum_{\tP\in \cG_V, \tQ\in \cG_W}|\tP|\cdot|\tQ|\cdot q^{-\nu_{P,Q}}.$$

To prove Theorem~\ref{thm: tensors}, it is natural to compare $|\tP|\cdot|\tQ|$ and $q^{-\nu_{P,Q}}$. We distinguish the following three cases:
\begin{itemize}
	\item First, the conjugacy class pair $(\tP, \tQ)$ satisfies the trivial bound $|\tP|\cdot|\tQ|\leq q^{n^2+m^2}$. In this case, we aim to show that $\nu_{P,Q}$ is \emph{considerably larger} than $n^2+m^2$; namely, $\nu_{P,Q}-(n^2+m^2)\geq \epsilon n^2-O(n)$ for some constant $\epsilon>0$, which is denoted by $\nu_{P,Q}\gg n^2+m^2$. As $m=C\cdot m+R$ for some constant $C$ and $R$, we only need to prove that the coefficient of $n^2$ in $\nu_{P,Q}$ is (strictly) larger than $1+C^2$.
	\item Second, \emph{$P$ and $Q$ are almost scalar}, i.e.~$(\tP, \tQ)$ satisfies an improved bound $|\tP|\cdot |\tQ|\leq q^{O(n)}$ (cf.~Lemma~\ref{lem:almost_scalar_conj_size}). In this case, we will show that $\nu_{P,Q}=\Omega(n^2)$.
	\item Third, \emph{$P$ and $Q$ are both scalar matrices}, and this contributes to the summation of $q$ in Equation \eqref{eq:fix_sum}.
\end{itemize}

We shall prove a series of propositions which leads to the proof of Theorem~\ref{thm: tensors}. Since conjugated $P,P'\in\GL(V)$ and conjugated $Q,Q'\in\GL(W)$ satisfy $|\Fix(P,Q)|=|\Fix(P', Q')|$, we shall assume without loss of generality that $P$ and $Q$ are in their rational normal forms. That is, $P$ is a block diagonal matrix with the blocks being the companion matrices of its elementary divisors, and similarly for $Q$. Meanwhile, since we consider $\Box=\{\wedge,\odot,\otimes\}$, we have $\frac{1}{2}n^2-\frac{1}{2}n\leq d\leq n^2$.

\paragraph{Bounding the largest block size.} We first show that, if (the rational normal form of) $P$ contains a ``large'' block, then $\nu_{P,Q}\gg n^2+m^2$. 
\begin{proposition}\label{prop:largest-P}
Let $V$ and $W$ be vector spaces over $\F_q$, with $\dim(V)=n$, $\dim(W)=m$, and $m=Cn+R$.
	Let $(P,Q)\in\GL(V)\times\GL(W)$. Suppose $P$'s rational normal form contains at least one block of size $>1+C+1/C$. Then $\nu_{P,Q}\gg n^2+m^2$. 
\end{proposition}
\begin{proof}
	Without loss of generality, we reorder the blocks of $P$ such that the largest block comes first. The largest block is associated with a subspace $U_1 \subseteq V$ of dimension $\ell=\dim(U_1)> 1+C+1/C$. Let $U_2$ be the complement of $U_1$ with respect to the rational normal form; that is, $U_1 \oplus U_2 = V$ is an invariant decomposition of $V$ with respect to $P$. We can arrange an ordered basis $(a_1, \dots, a_\ell)$ of $U_1$ and an ordered basis $(a_{\ell+1}, \dots, a_{n})$ of $U_2$, so that $P$ has the block matrix form in Figure~\ref{fig: rational normal form}.
	\begin{figure}[ht!]
	\begin{center}
		\tikzset{
			mathtable nodes/.style={
				align=center,
				minimum height=7mm,
				text depth=0.5ex,
				text height=2ex,
				inner xsep=0pt,
				outer sep=0pt
			},
		}
		\begin{tikzpicture}
			\matrix (m) [
			row sep=-\pgflinewidth,
			column sep=-\pgflinewidth,
			matrix of math nodes,
			nodes in empty cells,
			text width=7mm,
			nodes={
				mathtable nodes
			},
			]
			{
				&     &   &  & & & & & & & \\
				&     & a_1 & a_2 & a_3 & & & & & & \\
				& a_1 & 0 & & & & * & & & & \\
				& a_2 & 1 & 0 &  & & * & & & & \\
				& a_3 &   & 1 & 0 &  & * & & & &  \\
				&     &   &   & 1 & \ddots & * & & & & \\
				&     &   &   &   & \ddots & * & & & & \\
				& & & & & & & & & & \\
				&  & & & & & & & & & \\
				& & & & & & & & & & \\
				& & & & & & & & & & \\
			};
			\node[fit=(m-1-3)(m-1-7),mathtable nodes, fill=black!20]{$U_1$};
			\node[mathtable nodes, fill=black!10, minimum width=28mm, minimum height = 7mm] at (m-1-10.west) {$U_2$};
			\node[mathtable nodes, fill=black!20, minimum width=7mm, minimum height = 35mm] at (m-5-1) {$U_1$};
			\node[mathtable nodes, fill=black!10, minimum width=7mm, minimum height = 28mm] at (m-10-1.north) {$U_2$};
			\node[draw, rectangle, inner sep=0pt, fit={(m-3-3.north west) (m-7-7.south east)}] {};
			\node[draw, rectangle, inner sep=0pt, fit={(m-8-8.north west) (m-9-9.south east)}] {{\raisebox{-4mm}{\Huge $*$} }};
			\node[draw, rectangle, inner sep=0pt, fit={(m-10-10.north west) (m-11-11.south east)}] {{\raisebox{-4mm}{\Huge $*$} }};
			
		\end{tikzpicture}
		\caption{The rational normal form of $P$ with respect to the direct sum $U_1\oplus U_2$.}
			\label{fig: rational normal form}
	\end{center}
	\end{figure}
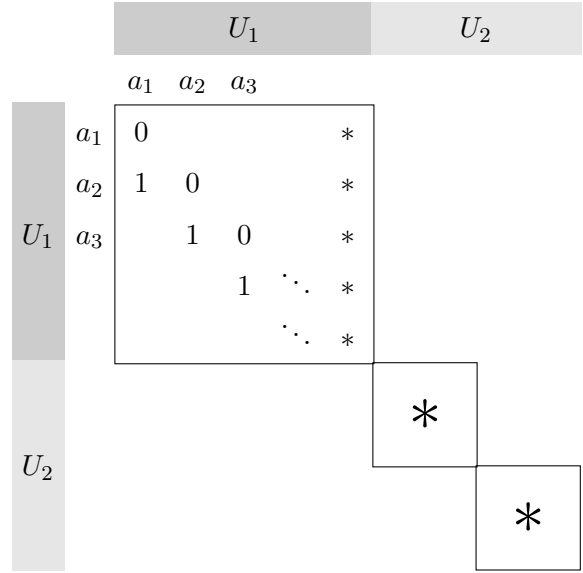
	
	Note that the image of $a_i$ under $P$ is $a_{i+1}$ for $i\in\{1,\dots,\ell-1\}$. Let $\{b_1,\dots,b_m\}$ be a linear basis of $W$ and $D_1=\lceil 1+C+1/C\rceil$. Consider the following linear independent vectors in $V\Box V\otimes W$:
	\[
	S=\{a_i\Box a_j\otimes b_k\mid i\in\{1,\dots,D_1-1\}, j\not\in\{1,\dots,D_1-1,\ell\}, k\in[m]\}.
	\] 
	Let $I_S\subseteq [n]\times [n]\times [m]$ be the index set of $S$. Note that $|S|=|I_S|= (D_1-1)\cdot(n-D_1)\cdot m=C(D_1-1)n^2+(D_1-1)(R-CD_1)n-D_1(D_1-1)R$. Pick an arbitrary vector $v\in\linspan(S)$ and write $v=\sum_{(i,j,k)\in I_S}\lambda_{i,j,k}\ a_i\Box a_j\otimes b_k$, its image under the action of $(P,Q)$ is 
	\[
	(P,Q)\circ v=\sum_{(i,j,k)\in I_S}\lambda_{i,j,k}\ Pa_i\Box Pa_j\otimes Qb_k=\sum_{(i,j,k)\in I_S}\lambda_{i,j,k}\ a_{i+1}\Box Pa_{j}\otimes Qb_k.
	\]
	Note that $Pa_{j}\not\in\linspan(\{a_1,\dots,a_{D_1-1}\})$. Thus, any nonzero vector $v\in\linspan(S)$ cannot be fixed by the action of $(P,Q)$. Therefore, $\linspan(S)\cap \Fix(P,Q)=\{0\}$. Since $\dim(\linspan(S))=|S|= C(D_1-1)n^2- O(n)$,
	we have
	\[
	\nu_{P,Q}\geq C(D_1-1)n^2-O(n).
	\]
	Recall that $m=Cn+R$ for some constant $C$ and $R$. As $D_1>1+C+1/C$, $C(D_1-1)-(1+C^2)>0$, we have $\nu_{P,Q}\gg n^2+m^2$.
\end{proof}

We can follow a similar strategy to bound the largest block size in $Q$. 

\begin{proposition}\label{prop:largest-Q}
	Let $(P,Q)\in\GL(V)\times\GL(W)$. Suppose $Q$'s rational normal form contains at least one block of size $>2C^2+3$. Then $\nu_{P,Q}\gg n^2+m^2$.
\end{proposition}
\begin{proof}
	Let $U\subseteq W$ be the subspace associated with the largest block of $Q$ in its rational normal form. Let $\{b_1,\dots,b_\ell\}$ be a basis of $U$ with $\ell>2C^2+3$. Let $\{c_1,\dots,c_d\}$ be a basis of $V\Box V$, and 
    set $D_2=\lceil2C^2+3\rceil$. We pick the following set of linear independent vectors:
	\[
	S=\{c_i\otimes b_k\mid i\in[d], k\in[\ell-1]\}.
	\]
	The size of $S$ is $(\ell-1)d\geq \frac{D_2-1}{2}n^2- O(n)$.
    Analogously to the proof of Proposition~\ref{prop:largest-P}, it can be verified that $\linspan(S) \cap \Fix(P,Q)=\{0\}$. Thus, $\nu_{P,Q}\geq \frac{D_2-1}{2}n^2- O(n)$. Since $D_2>2C^2+3$, we have $\frac{D_2-1}{2}-(1+C^2)>0$, so $\nu_{P,Q}\gg n^2+m^2$.
\end{proof}

\paragraph{Bounding the number of blocks of size $\geq 2$ (to go to almost diagonal).} From the above, if $P$ contains a block of size $>1+C+1/C$, or $Q$ contains a block of size $>2C^2+3$, then $\nu_{P,Q}$ is considerably larger than $n^2+m^2$. So it remains to deal with the setting when the block sizes in the rational normal forms of $P$ and $Q$ are upper bounded by the above constants, respectively. The next parameter to examine is \emph{the number of blocks of size $\geq 2$}.

\begin{proposition}\label{prop:number of block P}
	Let $(P,Q)\in\GL(V)\times\GL(W)$, where the block sizes in $P$'s rational normal form are upper bounded by $1+C+1/C$. Suppose $P$'s rational normal form has more than $C+1/C$ blocks of size $\geq 2$. Then $\nu_{P,Q}\gg n^2+m^2$.
\end{proposition}
\begin{proof}
	In the rational normal form of $P$, we order the blocks of size $\geq 2$ first, and then those of size $1$. Let $U_1, \dots, U_{E_1}$ be the subspaces associated with the first $E_1>C+1/C$ blocks of size $\geq 2$. For $i\in[E_1]$, let $a_i\in U_i$ be the first basis vector in the $i$th block of size $\geq 2$. As each block is of size $\leq 1+C+1/C$, the first $E_1$ blocks take at most $\ell\leq (1+C+1/C)E_1$ many linearly independent basis vectors. Let the remaining basis of $V$ be $\{a_{\ell+1},\dots,a_n\}$ according to $P$'s rational normal form. Let $\{b_1, \dots, b_m\}$ be a basis of $W$. Consider the following set of linearly independent vectors 
	\[
	S=\{a_i\Box a_j\otimes b_k\mid i\in[E_1],j\in\{\ell+1, \dots, n\},k\in[m]\}.
	\]
	Note that $|S|\geq E_1(n-\ell)m$. Then any non-zero $v\in \linspan(S)$ cannot be fixed, because any $a_i$ is sent to a vector that does not appear in $\linspan\{a_i\mid i\in[E_1]\cup\{\ell+1, \dots, n\}\}$.
	Since $\dim(\linspan(S))= E_1(n-\ell)m\geq E_1Cn^2-O(n)$ 
	and $E_1>C+1/C$, we have $E_1C-(1+C^2)>0$, so $\nu_{P,Q}\gg n^2+m^2$. 
\end{proof}

The following bound on the number of blocks of size $\geq 2$ can be shown for $Q$. Note that here we do not require an upper bound on the block sizes of $Q$'s rational norm form.
\begin{proposition}\label{prop:number of block Q}
	Let $(P,Q)\in\GL(V)\times\GL(W)$. Suppose $Q$'s rational normal form has more than $2C^2+2$ many blocks of size $\geq 2$. Then $\nu_{P,Q}\gg n^2+m^2$.
\end{proposition}
\begin{proof}
	Assume $Q$'s rational normal form has $E_2>2C^2+2$ many blocks of size $\geq 2$. Arrange an ordered basis of $W$ such that the blocks of size $\geq 2$ appear first. Let $b_1,\dots,b_{E_2}$ be the first basis vectors of these blocks of size $\geq 2$. Let $\{c_1,\dots,c_d\}$ be a basis of $V\Box V$. Let
	\[
	S=\{c_i\otimes b_k\mid i\in[d],~k\in[E_2]\}.
	\]
	It can be easily verified that $|S|= dE_2$ and $\linspan(S)\cap\Fix(P,Q)=\{0\}$. Since $E_2>2C^2+2$, we have $\nu_{P,Q}\geq dE_2\gg n^2+m^2$. 
\end{proof}

\paragraph{Bounding the numbers of different eigenvalues.} We are reduced to the setting where $P$'s (resp.\ $Q$'s) rational normal form has at most $C+1/C$ (resp.\ $2C^2+2$) many blocks of size $\geq 2$. Furthermore, the block sizes of $P$'s (resp.\ $Q$'s) rational normal form are at most $1+C+1/C$ (resp.\ $2C^2+3$). Therefore, $P$ and $Q$ are \emph{almost diagonal} in the following sense: There are ``large'' diagonal blocks in the rational normal forms of $P$ and $Q$. More precisely, Let $U_V$ be the subspace corresponding to the diagonal block in $P$'s rational normal form. By the discussions above, we have 
\begin{equation}\label{eq:s_bound}
	s_V:=\dim(U_V)\geq n-(C+1/C)(1+C+1/C).
\end{equation}
Let $U_W$ be the subspace corresponding to the diagonal block in $Q$'s rational normal form. By the discussions above, we have 
\begin{equation}\label{eq:t_bound}
	s_W:=\dim(U_W)\geq m-(2C^2+2)(2C^2+3).
\end{equation}
We will then restrict to consider the action of $(P,Q)$ on $U_V\Box U_V\otimes U_W$. 
The next parameter to examine is \emph{the number of distinct eigenvalues of $P|_{U_V}$ and $Q|_{U_W}$}. 

Let $\Spec(P|_{U_V})$ be the set of (distinct) eigenvalues of $P$ restricting on $U_V$. The following shows that if $|\Spec(P|_{U_V})|$ is large, then $\nu_{P,Q}$ is also considerably larger than $n^2+m^2$.
\begin{proposition}\label{prop: number of different eigenvalue P}
	Suppose $P\in\GL(V)$ and $Q\in\GL(W)$ satisfy the following:
	\begin{itemize}
		\item $P$'s (resp.\ $Q$'s) rational normal form has at most $C+1/C$ (resp.\ $2C^2+2$) many blocks of size $\geq 2$;
		\item the block sizes of $P$'s (resp.\ $Q$'s) rational normal form are at most $1+C+1/C$ (resp.\ $2C^2+3$).
	\end{itemize}
	Let $U_V$ (resp.\ $U_W$) be the subspace corresponding to the diagonal block in $P$'s (resp.\ $Q$'s) rational normal form with $\dim(U_V)=s_V\geq n-(C+1/C)(1+C+1/C)$ (resp.~$\dim(U_W)=s_W\geq m-(2C^2+2)(2C^2+3)$).
	If $|\Spec(P|_{U_V})|> 1+C+1/C$, then $\nu_{P,Q}\gg n^2+m^2$.
\end{proposition}
\begin{proof}
	Let $v_1, \dots, v_{s_V}\in V$ be linearly independent eigenvectors of $P$ such that $Pv_i=\alpha_iv_i$ for some $\alpha_1,\dots,\alpha_{s_V}\in\F_q$, and $w_1, \dots, w_{s_W}\in W$ be linearly independent eigenvectors such that $Qw_i=\beta_iw_i$ for some $\beta_1,\dots,\beta_{s_W}\in\F_q$. Then $v_i\Box v_j\otimes w_k\in \Fix(P,Q)$ if and only if $\alpha_i\alpha_j\beta_k=1$. 
	
	Without loss of generality, suppose $\alpha_1,\dots,\alpha_{F_1}$, where $F_1>1+C+1/C$, are all the distinct eigenvalues of $P|_{U_V}$ (after a suitable rearrangement of the eigenvectors). Then for any $j\in\{F_1+1, \dots, s_V\}$ and $k\in[s_W]$, there exists at most one $i\in[F_1]$, such that $\alpha_i\alpha_j\beta_k=1$. Therefore, within $\{v_i\Box v_j\otimes w_k\mid i\in[F_1], j\in\{F_1+1, \dots, s_V\}, k\in[s_W]\}$, at most $(s_V-F_1)s_W$ many vectors are fixed by the action of $(P,Q)$. By taking those non-fixed vectors, we obtain a set $S$ of linearly independent vectors in $V\Box V\otimes W$, such that $\linspan(S)\cap\Fix(P,Q)=\{0\}$. The size of $S$ can be lower bounded by  
	\[
	\begin{split}
		|S|&\geq ({F_1}-1)\cdot (s_V-{F_1})\cdot s_W\\
		&\geq (F_1-1)\cdot (n-(C+1/C)(1+C+1/C)-F_1)\cdot (m-(2C^2+2)(2C^2+3))\\
		&\geq (F_1-1)Cn^2-O(n).
	\end{split}
	\] 
	Since $F_1>1+C+1/C$, we have $(F_1-1)C-(1+C^2)>0$, so $\nu_{P,Q}\geq |S|\gg n^2+m^2$.
\end{proof}

Let $\Spec(B|_{U_W})$ be the set of distinct eigenvalues of $Q$ restricting on $U_W$. Analogously to Proposition~\ref{prop: number of different eigenvalue P}, we have the following.
\begin{proposition}\label{prop: number of different eigenvalue Q}
	Suppose $P\in\GL(V)$ and $Q\in\GL(W)$ satisfy the conditions in~Proposition~\ref{prop: number of different eigenvalue P}. If $|\Spec(B|_{U_W})|> 2C^2+3$, then $\nu_{P,Q}\gg n^2+m^2$.
\end{proposition}
\begin{proof}
	The proof is analogous to that of Proposition~\ref{prop: number of different eigenvalue P}. The difference lies in that the number of linearly independent vectors in $V\Box V\otimes W$, which is non-fixed by the action of $(P,Q)$, is larger than $s_V\cdot (2C^2+3-1)\gg n^2+m^2$.
\end{proof}

\paragraph{There exists a dominant eigenvalue (to go to almost scalar).} We are now in the setting that not only $P$ and $Q$ are almost diagonal, but the numbers of distinct eigenvalues for $P$ and $Q$ are bounded by $1+C+1/C$ and $2C^2+3$, respectively. We now examine \emph{the multiplicities of these eigenvalues}. 

For an eigenvalue $\alpha$ (resp.\ $\beta$) of $P$ (resp.\ $Q$), let $\mul(\alpha)$ (resp.\ $\mul(\beta)$) be the multiplicity of $\alpha$ (resp.\ $\beta$). 
We shall, for the first time, consider $P$ and $Q$ simultaneously.
\begin{proposition}\label{prop: bounded}
	Suppose $P\in\GL(V)$ and $Q\in\GL(W)$ satisfy the conditions in Proposition~\ref{prop: number of different eigenvalue P}, in addition that $|\Spec(P|_{U_V})|\leq 1+C+1/C$ and $\Spec(Q|_{U_W})|\leq 2C^2+3$. Let $\alpha\in\Spec(P|_{U_V})$ be an eigenvalue with the largest multiplicity in $P$, and $\beta\in\Spec(Q|_{U_W})$ be an eigenvalue with the largest multiplicity in $Q$. Suppose there are more than $(C+1/C)(1+C+1/C)(2C^2+3)$ many eigenvalues of $P|_{U_V}$ that are not $\alpha$, or there are more than $(2C^2+2)(1+C+1/C)^2$ many eigenvalues of $Q|_{U_W}$ that are not $\beta$, then $\nu_{P,Q}\gg n^2+m^2$.
\end{proposition}
\begin{proof}
	Let $U_\alpha\subseteq U_V$ and $U_\beta\subseteq U_W$ be the eigenspaces associated with $\alpha$ and $\beta$, respectively. Note that $\dim(U_\alpha)=\mul(\alpha)\geq s_V/(1+C+1/C)$, $\dim(U_\beta)=\mul(\beta)\geq s_W/(2C^2+3)$, and $\dim(U_{\alpha}\Box U_{\alpha})\geq \binom{s_V/(1+C+1/C)}{2}$. We distinguish between the following cases. 
	\begin{enumerate}
		\item If $\alpha^2\beta\neq 1$, then $(P,Q)$ can not fix any vector of the form $v_1\Box v_2\otimes w$, where $v_1,v_2\in U_\alpha$ are linearly independent eigenvectors of $P$ with eigenvalue $\alpha$, and $w\in U_\beta$ is any eigenvector of $Q$ with respect to the eigenvalue $\beta$. Via the same argument for the proof of Proposition~\ref{prop: number of different eigenvalue P}, we have $\nu_{P,Q}\geq \dim(U_\alpha\Box U_{\alpha}) \cdot \dim(U_\beta)=\Omega(n^3)$, which is a cubic function of $n$, therefore considerably larger than the quadratic $n^2+m^2$.
		\item If $\alpha^2\beta=1$, consider the following two cases.
		\begin{enumerate}
			\item If there are $H_1>(C+1/C)(1+C+1/C)(2C^2+3)$ many eigenvalues of $P|_{U_V}$ which are not $\alpha$, then for any eigenvalue $\alpha'\neq \alpha$ of $P|_{U_V}$,  $\alpha\alpha'\beta\neq 1$. Via the same argument for the proof of Proposition~\ref{prop: number of different eigenvalue P}, we have 
			\[
			\nu_{P,Q}\geq \dim(U_\alpha)\cdot H_1\cdot \dim(U_\beta)>(C+1/C)s_Vs_W\gg n^2+m^2.
			\] 
			\item If there are $H_2>(2C^2+2)(1+C+1/C)^2$ many eigenvalues of $Q|_{U_W}$ which are not $\beta$, for any eigenvalue $\beta'\neq \beta$ of $Q|_{U_W}$, $\alpha^2\beta'\neq 1$. We then have 
			\[
			\nu_{P,Q}\geq \dim(U_\alpha)^2 \cdot H_2>(C^2+1)s_V^2-O(s_V)\gg n^2+m^2.
			\]
		\end{enumerate}
	\end{enumerate}
	This concludes the proof.
\end{proof}

\paragraph{$P$ and $Q$ are almost scalar.} We now in the case that $P$ and $Q$ are \emph{almost scalar}: All but at most $(C+1/C)(1+C+1/C)(2C^2+4)$ (resp.\ $(2C^2+2)((1+C+1/C)^2+2C^2+3)$) eigenvalues of $P$ (resp.\ $Q$) is $\alpha$ (resp.\ $\beta$).
By Lemma~\ref{lem:almost_scalar_conj_size}, $|\conjc(P)|$ and $|\conjc(Q)|$ are upper bounded by $q^{O(n)}$ and $q^{O(m)}=q^{O(n)}$, respectively.

We now show that \emph{if one of $P$ and $Q$ is not fully scalar}, then $\nu_{P,Q}$ is not too small. 
\begin{proposition}\label{prop:almost-scalar}
	Suppose $P\in\GL(V)$ and $Q\in\GL(W)$ are \emph{almost scalar}: All but at most $(C+1/C)(1+C+1/C)(2C^2+4)$ (resp.\ $(2C^2+2)((1+C+1/C)^2+2C^2+3)$) eigenvalues of $P$ (resp.\ $Q$) is $\alpha$ (resp.\ $\beta$). If at least one of them is not fully scalar. Then $\nu_{P,Q}=\Omega(n^2)$. 
\end{proposition}
\begin{proof}
 $U_\alpha\subseteq V$ and $U_\beta\subseteq W$ be the corresponding eigenspaces of $\alpha$ and $\beta$, respectively. If $\alpha^2\beta\neq 1$, the claim holds by following Case $1$ in the proof of Proposition~\ref{prop: bounded}. 
	
	Now suppose $\alpha^2\beta=1$. We distinguish between two cases.
	\begin{enumerate}
		\item[(a)] Suppose $Q$ is not fully scalar. If $Q$'s rational normal form is not diagonal, let $w\in W$ be the first basis vector corresponding to a block of size $\geq 2$. If $Q$ is diagonal, let $w\in W$ be an eigenvector with eigenvalue $\neq \beta$. Consider $S=\{ v\Box v'\otimes w\mid v, v'\in U_\alpha\}$. It can be verified that $\dim(\linspan (S))=|S|=\dim(U_\alpha\Box U_{\alpha})$ and $\linspan(S)\cap \Fix(P,Q)=\{0\}$. Therefore, $\nu_{P,Q}\geq |S|=\Omega(n^2)$.
		\item[(b)] Suppose $P$ is not fully scalar. If the rational normal form of $P$ is not diagonal, let $v\in V$ be the first basis vector corresponding to a block of size $\geq 2$. If $P$ is diagonal, let $v\in V$ be an eigenvector of eigenvalue $\neq \alpha$. Consider $S=\{ v\Box v'\otimes w\mid v'\in U_\alpha, w\in U_\beta\}$. It can be verified that $\dim(S)=\dim(U_\alpha)\dim(U_\beta)$ and $\linspan(S)\cap \Fix(P,Q)=\{0\}$. Therefore, $\nu_{P,Q}\geq |S|=\Omega(n^2)$. \qedhere
	\end{enumerate}
\end{proof}

Summarising the above results, we are ready to prove Theorem~\ref{thm: tensors}:
\begin{proof}[Proof of Theorem~\ref{thm: tensors}]
	Let $d=\dim(V\Box V)$. The conjugacy class pairs in $\GL(V)\times\GL(W)$ fall into one of the following three classes. 
	\begin{itemize}
		\item For non-zero $\alpha, \beta\in\F_q$, let $\cG_{\alpha,\beta}$ be the conjugacy class pair corresponding to  $P =\alpha I_n$ and $Q=\beta I_m$. In this case, $\Fix(P,Q)=\{0\}$ if $\alpha^2\beta\neq 1$, and $\Fix(P,Q)=V\Box V\otimes W$ if $\alpha^2\beta=1$. Because there are $q-1$ choices of $(\alpha,\beta)$ satisfying $\alpha^2\beta=1$, we have 
		\[
		\sum_{(P,Q)\in\cG_{\alpha,\beta}}|\Fix(P,Q)|\cdot q^{-dm}=q-1.
		\]
		\item Let $\cG_\mathit{AS}$ be the conjugacy class pairs corresponding to almost scalar matrices $(P,Q)$ satisfying that there exists an eigenvalue of $P$ (resp.\ $Q$) with geometric multiplicity $(C+1/C)(1+C+1/C)(2C^2+4)$ (resp.\ $(2C^2+2)((1+C+1/C)^2+2C^2+3)$).
		By Proposition~\ref{prop:almost-scalar}, $\nu_{P,Q}=\Omega(n^2)$ if $(P,Q)\in\cG_{AS}$. The number of conjugacy classes is bounded by $O(q^n)$ (cf.~Lemma~\ref{lem:total_num_conj_classes}) and the size of each conjugacy class is bounded by $q^{O(n)}$ (cf.~Lemma~\ref{lem:almost_scalar_conj_size}). We have
		\[
		\sum_{(P,Q)\in\cG_\mathit{AS}}|\Fix(P,Q)|\cdot q^{-dm}\leq O(q^n)\cdot q^{O(n)}\cdot q^{-\nu_{P,Q}}\leq q^{-\Omega(n^2)}.
		\]
		\item Let $\cG_\mathit{NAS}$ be the set of the rest conjugacy class pairs. By Proposition~\ref{prop:largest-P}--\ref{prop: bounded} and the number of conjugacy classes is bounded by $O(q^n)$ (cf.~Lemma~\ref{lem:total_num_conj_classes}), we have
		\[
		\sum_{(P,Q)\in\cG_{NAS}}|\Fix(P,Q)|\cdot q^{-dm}\leq O(q^n)\cdot q^{n^2+m^2}\cdot q^{-\nu_{P,Q}}\leq q^{-\Omega(n^2)},
		\]
		where we utilize the trivial bound $q^{n^2}$ (resp.\ $q^{m^2}$) for the size of each conjugacy class in $\GL(V)$ (resp.\ $\GL(W)$).  
	\end{itemize}
	Combine the above three inequalities, we have that
	\[
	\sum_{(P,Q)\in\GL(V)\times\GL(W)}|\Fix((P,Q))|\cdot q^{-dm}\leq q-1+\frac{1}{q^{\Omega(n^2)}}.\qedhere
	\]
\end{proof}

\subsection{Proof of Theorem~\ref{thm:space} and Corollary~\ref{cor:fraction}}\label{subsec:cor-space}
Theorem~\ref{thm: tensors} leads to the following result on subspaces of $V\Box V$. 
\begin{theorem}\label{thm:space-technical}
	Suppose $V\cong\F_q^n$ and $m=C\cdot n+R$ for constants $C$ and $R$. Let $K=q^{m^2}/|\GL(m,q)|$, and $\Box\in\{\otimes,\odot,\wedge\}$. Then 
	$$g(V\Box V, m)\leq K\cdot (1+\frac{1}{q^{\Omega(n^2)}})\cdot \frac{|\Gr(V\Box V, m)|}{|\PGL(V)|}.$$
\end{theorem}
Theorem~\ref{thm:space-technical} implies Theorem~\ref{thm:space} by taking $\Box=\wedge$ and noting that $g(V\Box V, m)/(K\cdot \frac{|\Gr(V\Box V, m)|}{|\PGL(V)|})\to 1$ as $n\to \infty$. For the furthermore part of Theorem~\ref{thm:space} (for $n, q\to\infty$), we can use the fact that $K\leq 1+\frac{1}{q-2}$.
\begin{proof}
	Let $O$ be the set of orbits of $V\Box V\otimes W$ under $\GL(V)\times\GL(W)$. Let $O'$ be the set of orbits of $\Gr(V\Box V, m)$ under $\GL(V)$. 
	By Theorem~\ref{thm: tensors}, $|O|\leq (q-1+\frac{1}{q^{\Omega(n^2)}})\cdot \frac{q^{dm}}{|\GL(V)|\cdot|\GL(W)|}$, where $d=\dim(V\Box V)$. Note that every $t\in V\Box V\otimes W$ naturally gives rise to $U_t\in \Gr(V\Box V, \leq m)$, which consists of subspaces of $V\Box V$ of dimension $\leq m$. 
    Let $(V\Box V\otimes W)^\fdim$ be the set of $t\in V\Box V\otimes W$ such that $U_t\in \Gr(V\Box V, m)$. Let $O^\fdim$ be the set of orbits of $(V\Box V\otimes W)^\fdim$ under $\GL(V)\times\GL(W)$. It is clear that $O^\fdim\subseteq O$, so 
	\begin{equation}\label{eq:gr1}
		|O^\fdim|\leq |O|\leq (q-1+\frac{1}{q^{\Omega(n^2)}})\cdot \frac{q^{dm}}{|\GL(V)|\cdot|\GL(W)|}=(1+\frac{1}{q^{\Omega(n^2)}})\cdot \frac{q^{dm}}{|\PGL(V)|\cdot|\GL(W)|}.
	\end{equation}
	
	Then observe that $t_1,t_2\in (V\Box V\otimes W)^\fdim$ are in the same orbit under $\GL(V)\times\GL(W)$ if and only if $U_{t_1}$ and $U_{t_2}$ are in the same orbit of $\GL(V)$ acting on $\Gr(V\Box V, m)$. It follows that $|O'|=|O^\fdim|$. Therefore, 
	\begin{equation}\label{eq:gr2}
	q^{dm}/|\GL(W)|\leq K\cdot q^{dm-m^2}\leq K\cdot |\Gr(V\Box V, m)|
	\end{equation}
	where $K=q^{m^2}/|\GL(m,q)|$. Plugging Equation~\eqref{eq:gr2} into Equation~\eqref{eq:gr1} we get the desired result. 
\end{proof}
For $V\cong\F_q^n$, we have correspondences between orbits of $\Gr(V\otimes V,m)$ under the natural action of $\GL(V)$ and orbits of $\Gr(\M(n,q),m)$ under the congruence action of $\GL(n, q)$, as well as between $\Gr(V\odot V,m)$ and $\Gr(\S(n,q),m)$, and $\Gr(V\wedge V,m)$ and $\Gr(\Lambda(n,q),m)$.
Then Theorem~\ref{thm:space-technical} yields the following corollary, which covers Corollary~\ref{cor:fraction}.
\begin{corollary}\label{cor:fraction-technical}
Suppose $m=Cn+R$ and let $K=q^{m^2}/|\GL(m,q)|$. When $n$ is large enough, all but $\frac{1}{q^{\Omega(n^2)}}$ fraction of $m$-dimensional subspaces of $\Lambda(n, q)$ (resp.\ $\M(n, q)$, $\S(n, q)$) have their automorphism group order upper bounded by $K\cdot (q-1)$.
\end{corollary}
\begin{proof}
	By Theorem~\ref{thm:space-technical}, there exists a constant $C$, such that when $n$ is large enough, the average order of $\Aut(\cA)$ over $\cA\in \Gr(\Lambda(n, q),m)$ is upper bounded by $K\cdot(q-1+\frac{1}{q^{Cn^2}})$. 
	
	Take any constant $0<D<C$. We claim that all but at most $1/q^{Dn^2}$ fraction of $m$-dimensional subspaces $\cA$ of $\Lambda(n, q)$ have $\Aut(\cA)\leq K\cdot (q-1)$. If not, we have at least $1/q^{Dn^2}$ fraction of subspaces have their automorphism groups of order $\geq K\cdot(q-1)+1$. This gives that the average order is at least $K\cdot(q-1)\cdot (1-1/q^{Dn^2})+(K\cdot (q-1)+1)\cdot 1/q^{Dn^2}=K\cdot(q-1)+1/q^{Dn^2}=K\cdot(q-1+1/q^{Dn^2+2})$. Here we make use of the fact that $K<1/0.288788<4$. As we have set $D<C$,  $K\cdot(q-1+1/q^{Dn^2+2})>K\cdot(q-1+1/q^{Cn^2})$ for large enough $n$, a contradiction to Theorem~\ref{thm:space-technical}. 
\end{proof}

\section{Application to finite group enumeration}\label{sec:aux}
\subsection{From groups to matrix spaces}

In \cite{Hig60}, Higman observed the following for $\higp$ when $p>2$. Let $\higp(n, m)$ be the set of (isomorphic classes of) $p$-groups of Frattini class $2$, with the Frattini quotient isomorphic to $\Z_p^n$, and the Frattini subgroup isomorphic to $\Z_p^m$. Let $f_{\higp(n, m)}$ be the number of groups in $\higp(n, m)$. 
\begin{proposition}[{\cite[Theorem 2.2]{Hig60}}]\label{prop:higp_lin}
	Let $p>2$ be a prime. Let $V\cong \F_p^n$ and let $V\wedge V$ be its exterior square. Then $f_{\higp(n, m)}$ is equal to the number of isomorphism classes of dimension-$m$ subspaces of $V\wedge V\oplus V$ under $\GL(V)$.
\end{proposition}
The original formulation of \cite[Theorem 2.2]{Hig60} was for \emph{co}dimension-$m$ subspaces. Proposition~\ref{prop:higp_lin} is obtained by examining the action of $\GL(V)$ on the dual space of $V\wedge V\oplus V$.

Similarly, let $\twop(n, m)$ be the set of (isomorphic classes of) $p$-groups of class $2$ and exponent $p$, with the commutator quotient isomorphic to $\Z_p^n$, and the commutator subgroup isomorphic to $\Z_p^m$. Let $f_{\twop(n, m)}$ be the number of groups in $\twop(n, m)$. The following is an easy consequence of the method as in \cite{Hig60}.
\begin{proposition}\label{prop:twop_lin}
	Let $p>2$ be a prime. Let $V\cong \F_p^n$ and $V\wedge V$ be its exterior square. Then $f_{\twop(n, m)}$ is equal to the number of isomorphism  classes of dimension-$m$ subspaces of $V\wedge V$ under $\GL(V)$.
\end{proposition}

Propositions~\ref{prop:higp_lin} and \ref{prop:twop_lin} give the following approach for Theorem~\ref{thm:group-enum}. Note that 
\[
f_\higp(p^\ell)=\sum_{1\leq n\leq \ell}f_{\higp(n, \ell-n)}
\]
and
\[
f_\twop(p^\ell)=\sum_{1\leq n\leq \ell}f_{\twop(n, \ell-n)}.
\]
If we can prove that $\max_{n\in[\ell]} f_{\higp(n, \ell-n)}=p^{\frac{2}{27}\ell^3-\frac{4}{9}\ell^2+\Theta(\ell)}$ and $\max_{n\in[\ell]} f_{\twop(n, \ell-n)}=p^{\frac{2}{27}\ell^3-\frac{2}{3}\ell^2+\Theta(\ell)}$, then Theorem~\ref{thm:group-enum} would follow. To bound $f_{\higp(n, \ell-n)}$ and $f_{\twop(n, \ell-n)}$, by Proposition~\ref{prop:higp_lin} and Proposition~\ref{prop:twop_lin} we can examine the action of $\GL(V)$ on subspaces of $V\wedge V\oplus V$ and  $V\wedge V$. 

\subsection{Enumerating orbits of \texorpdfstring{$\Gr(V\wedge V\oplus V,m)$}{}}
Theorem~\ref{thm: tensors} provides a bound on the number of orbits of $V\wedge V$ under the action of $\GL(V)$. Now we derive a similar bound on the number of orbits of $V\wedge V\oplus V$ under the action of $\GL(V)$.
\begin{theorem}\label{thm: tensors_frattini}
	Suppose $V\cong \F_q^n$ and $W\cong \F_q^m$, where $m=\lceil C\cdot n+R\rceil$ for constants $C$ and $R$. Let $K=q^{m^2}/|\GL(m, q)|$. Then 
	$$g((V\wedge V\oplus V)\otimes W)\leq (q-1+\frac{1}{q^{\Omega(n^2)}})\cdot \frac{q^{(\binom{n}{2}+n)\cdot m}}{|\GL(V)|\cdot|\GL(W)|},$$ and 
	$$g(V\wedge V\oplus V, m)\leq K\cdot (q-1+\frac{1}{q^{\Omega(n^2)}})\cdot \frac{|\Gr(V\wedge V\oplus V, m)|}{|\GL(V)|}.$$
\end{theorem}
\begin{proof}
	The first inequality is equivalent to the following due to Burnside's lemma:
	\begin{equation}\label{eq:fix_sum_frattini}
		\sum_{(P,Q)\in\GL(V)\times\GL(W)}|\Fix((P,Q), (V\wedge V\oplus V)\otimes W)|\leq (q-1+\frac{1}{q^{\Omega(n^2)}})\cdot q^{(\binom{n}{2}+n)\cdot m},
	\end{equation}
	where 
	$$\Fix((P,Q), (V\wedge V\oplus V)\otimes W)=\{t\in (V\wedge V\oplus V)\otimes W\mid (P,Q)\circ t=t\}.$$
	We first have the following simple observation.
	\begin{observation}\label{obs:frattini}
		For $P\in\GL(V)$ and $Q\in\GL(W)$,
		\[
		|\Fix((P,Q), (V\wedge V\oplus V)\otimes W)|\leq |\Fix((P,Q), V\wedge V\otimes W)|\cdot |V\otimes W|
		\]
	\end{observation}
	\begin{proof}
		Choose an arbitrary vector $t\in \Fix((P,Q), (V\wedge V\oplus V)\otimes W)$. Let $t=t_1+t_2$, where $t_1\in V\wedge V\otimes W$ and $t_2\in V\otimes W$. Recall that the action of $P$ on $V\wedge V\oplus V$ sends $(v_1\wedge v_2, v_3)\in V\wedge V\oplus V$ to $(Pv_1\wedge Pv_2, Pv_3)$. That is, the induced action of $P$ on $V\wedge V\oplus V$ splits into the actions of $P$ on the invariant subspaces of $V\wedge V$ and $V$. Therefore,
		\[
		(P,Q)\circ t_1+(P,Q)\circ t_2=(P,Q)\circ t=t=t_1+t_2.
		\]
		This implies that $t_1\in \Fix((P,Q), V\wedge V\otimes W)$ and $t_2\in\Fix((P,Q), V\otimes W)$. Since $|\Fix((P,Q), V\otimes W)|\leq |V\otimes W|$, the claimed result follows.
	\end{proof}
	Therefore, we have 
	\begin{eqnarray*}
		&&\sum_{(P,Q)\in\GL(V)\times\GL(W)}|\Fix((P,Q), (V\wedge V\oplus V)\otimes W)|\\
		&\leq &\sum_{(P,Q)\in\GL(V)\times\GL(W)}\big(|\Fix((P,Q), V\wedge V\otimes W)|\cdot |V\otimes W|\big)\\
		&=&\left(\sum_{(P,Q)\in\GL(V)\times\GL(W)}|\Fix((P,Q), V\wedge V\otimes W)|\right)\cdot |V\otimes W|\\
		&\leq & (q-1+\frac{1}{q^{\Omega(n^2)}})\cdot q^{\binom{n}{2}\cdot m}\cdot q^{nm}\\
		&=&(q-1+\frac{1}{q^{\Omega(n^2)}})\cdot q^{(\binom{n}{2}+n)\cdot m},
	\end{eqnarray*}
	where the second $\leq$ is due to Theorem~\ref{thm: tensors} by replacing $\Box$ by $\wedge$.
	
	The result for $g((V\wedge V\oplus V), m)$ follows by the above and the same argument in the proof of Theorem~\ref{thm:space} (proof in Section~\ref{subsec:cor-space}).
\end{proof}

\subsection{Asymptotic bounds on \texorpdfstring{$f_\twop$}{} and \texorpdfstring{$f_\higp$}{}}\label{subsec:group-enum}

In this subsection we prove Theorem~\ref{thm:group-enum}. Theorems~\ref{thm:space-technical} and~\ref{thm: tensors_frattini}, together with Propositions~\ref{prop:higp_lin} and~\ref{prop:twop_lin}, give the desired bounds in Theorem~\ref{thm:group-enum} for $\twop$ and $\higp$ where $p>2$. The case of $p=2$ for $p$-groups of Frattini class $2$ can be dealt with certain technical twists. 

In the following, we handle $\twop$, $\higp$ with $p>2$, and $\higp$ with $p=2$ one by one.

\paragraph{The case of $\twop$.}  Recall that $f_{\twop}(p^\ell)=\sum_{1\leq n\leq \ell}f_{\twop(n, \ell-n)}$. The following statement helps us to pin down the range of $n$ for which the maximum is achieved.

\begin{proposition}\label{prop:max_range}
	There exist an integer $B\in \N$ such that when $\ell>B$, $\max_{n\in[\ell]} f_{\twop(n, \ell-n)}$ is achieved at some $n$ satisfying $\frac{1}{2}\ell\leq n\leq \frac{5}{6}\ell$.
\end{proposition}
\begin{proof}
	By \cite{Hig60} (see \cite[Theorem 19.3]{BNV07}), we have $p^{m(\binom{n}{2}-m)-n^2}\leq f_{\twop(n, \ell-n)}\leq p^{m(\binom{n}{2}-m+1)}$. Let $g(n, \ell):=p^{m(\binom{n}{2}-m)-n^2}$ and $h(n, \ell):=p^{m(\binom{n}{2}-m+1)}$, where $m=\ell-n$. 
	
	First, setting $n=\lceil \nicefrac{2}{3}\cdot\ell\rceil$, we see that $\log_p g(\lceil \nicefrac{2}{3}\cdot\ell\rceil, \ell)=\frac{2}{27}\ell^3+g'(\ell)$ where $g'(\ell)$ is a quadratic polynomial. 
	
	Second, setting $n=\lceil \nicefrac{1}{2}\cdot\ell\rceil$, we see that $\log_p h(\lceil \nicefrac{1}{2}\cdot\ell\rceil, \ell)=\frac{1}{16}\ell^3+h'(\ell)$ where $h'(\ell)$ is a quadratic polynomial. 
	
	Third, setting $n=\lfloor \nicefrac{5}{6}\cdot\ell\rfloor$, we see that $\log_p h(\lfloor \nicefrac{5}{6}\cdot\ell\rfloor, \ell)=\frac{25}{432}\ell^3+h''(\ell)$ where $h''(\ell)$ is a quadratic polynomial.
	
	The above suggest that when $\ell$ is large enough, $g(\lceil \nicefrac{2}{3}\cdot\ell\rceil, \ell)$ is larger than $h(\lceil \nicefrac{1}{2}\cdot\ell\rceil, \ell)$ and $h(\lfloor \nicefrac{5}{6}\cdot\ell\rfloor, \ell)$. Furthermore, it can be checked that $h$, when $\ell$ is fixed, is an increasing function from $n=0$ to $n=\lceil \nicefrac{1}{2}\cdot\ell\rceil$, and a decreasing function from $n=\lfloor \nicefrac{5}{6}\cdot\ell\rfloor$ to $\ell$. It follows that the maximum of $f_{\twop(n, \ell-n)}$ is achieved at some $\frac{1}{2}\ell\leq n\leq \frac{5}{6}\ell$.
\end{proof}

Proposition~\ref{prop:max_range} shows that $\max_{n\in[\ell]} f_{\twop(n, \ell-n)}$ is achieved at the range $\frac{1}{2}\ell\leq n\leq \frac{5}{6}\ell$. Within this range, $n$ is linearly correlated with $m=\ell-n$, so by Proposition~\ref{prop:twop_lin} and Theorem~\ref{thm:space-technical} (by letting $\Box=\wedge$), we have
\begin{eqnarray*}
	f_{\twop(n, \ell-n)} & = & g(\F_p^n\wedge \F_p^n, \ell-n) \\
	&\leq & K\cdot (p-1+\frac{1}{p^{\Omega(n^2)}})\cdot \frac{|\Gr(\F_p^n\wedge \F_p^n, m)|}{|\GL(n, p)|} \\
	&\leq & K\cdot (p-1+\frac{1}{p^{\Omega(n^2)}})\cdot \frac{p^{n^2}}{|\GL(n, p)|}\cdot \frac{p^{(\binom{n}{2}-m+1)\cdot m}}{p^{n^2}} \\
	& = & p^{\binom{n}{2}\cdot m-m^2-n^2+m+O(1)}.
\end{eqnarray*}
The question then becomes to maximise $\binom{n}{2}\cdot m-m^2-n^2+m$, where $m=\ell-n$. Following \cite{Hig60}, we set $n=(2\ell - \delta)/3$ and $m=(\ell+\delta)/3$ where $\delta\in\Q$ is a constant with $|\delta|<1$, as these maximise the coefficients of the cubic term. Putting these into $\binom{n}{2}\cdot m-m^2-n^2+m$, we get 
$$
\frac{2}{27}\ell^3-\frac{2}{3}\ell^2-\frac{1}{18}(\delta^2-3\delta-6)\ell+\frac{1}{54}(\delta^3-9\delta^2+18\delta).
$$
The coefficients of $\ell^3$ and $\ell^2$ are then determined. The coefficient of the linear term is easily seen to be upper bounded by a constant ($\frac{33}{72}$ in this case), so the result follows.

\paragraph{The case of $\higp$ with $p>2$.} When $p>2$, the statement for $f_{\higp}$ in Theorem~\ref{thm:group-enum} can be done similarly as a consequence of Theorem~\ref{thm: tensors_frattini} and Proposition~\ref{prop:higp_lin}. Indeed, we get $f_{\higp(n, m)}\leq p^{\binom{n}{2}\cdot m+nm-m^2-n^2+m+O(1)}$, and a similar optimisation (by setting $n=(2\ell - \delta)/3$ and $m=(\ell+\delta)/3$) gives the desired result.

\paragraph{The case of $\higtwo$, i.e. $\higp$ with $p=2$.} When $p=2$, we need to exploit the structure of $2$-groups of Frattini class $2$ in more detail. We first recall some results from \cite{Hig60}, though we shall mostly refer to \cite{BNV07} due to its more detailed exposition. 

Let $F_{\Phi\text{-}2, p, n}$ be the relatively free $p$-group of Frattini class $2$ with $n$ generators. Let $x_1, \dots, x_n$ be a set of generators of $F_{\Phi\text{-}2, p, n}$. Let $P$ be the group generated by $x_i^2$, $i\in[n]$. Let $D$ be the commutator subgroup generated by $[x_i, x_j]$, $1\leq i<j\leq n$. By \cite[pp. 24]{BNV07}, $x_i^p$ and $[x_i,x_j]$ form a minimal generating set of $\Phi(F_{\Phi\text{-}2, p, n})$; that is, there is no non-trivial relation between them. Therefore, $\Phi(F_{\Phi\text{-}2, p, n})\cong \Z_p^{\binom{n}{2}+n}$, and $F_{\Phi\text{-}2, p, n}/\Phi(F_{\Phi\text{-}2, p, n})\cong \Z_p^{n}$. In particular, we can identify $\Phi(F_{\Phi\text{-}2, p, n})$ as a vector space over $\F_p$, and subgroups of $\Phi(F_{\Phi\text{-}2, p, n})$ correspond to subspaces of $\F_p^{\binom{n}{2}+n}$.

By \cite[Lemma 4.1]{BNV07}, any $p$-group of Frattini class $2$ is isomorphic to a quotient group of $F_{\Phi\text{-}2, p, n}$ with respect to a subgroup of $\Phi(F_{\Phi\text{-}2, p, n})$. Let $G_1$ and $G_2$ be two $p$-groups of Frattini class $2$, with $G_i/\Phi(G_i)\cong \Z_p^n$ for $i=1, 2$, so for $i\in[2]$, $G_i\cong F_{\Phi\text{-}2, p, n}/N_i$ where $N_i\leq \Phi(F_{\Phi\text{-}2, p, n})$. By \cite[Lemma 4.3]{BNV07}, $G_1$ and $G_2$ are isomorphic if and only if there exists $T\in\Aut(F_{\Phi\text{-}2, p, n}/\Phi(F_{\Phi\text{-}2, p, n}))\cong \GL(n, p)$ such that the induced action of $T$ sends $N_1$ to $N_2$ as subspaces of $\Phi(F_{\Phi\text{-}2, p, n})$. 

We then need to examine the induced action of $\Aut(F_{\Phi\text{-}2, p, n}/\Phi(F_{\Phi\text{-}2, p, n}))$ on $\Phi(F_{\Phi\text{-}2, p, n})$. When $p>2$, the situation is easier, as $P$ and $D$ are both invariant under this action, and can be seen equivalent to the natural action $\GL(V)$ on $V\wedge V\oplus V$ where $V\cong\F_p^n$ \cite[Theorem 2.2]{Hig60}. 

When $p=2$, we have $(x_ix_j)^2=x_i^2x_j^2[x_j,x_i]$. As $[x_j,x_i]=[x_i,x_j]^{-1}=[x_i,x_j]$ (using the fact that in $F_{\Phi\text{-}2, 2, n}$, $[x_i,x_j]^2=\id$), this leads to $(x_ix_j)^2=x_i^2x_j^2[x_i,x_j]$. That is, $P$ is no longer invariant under this action. Fortunately, the commutator group $D$ is still invariant under this action, and the induced action on $D$ is still equivalent to the natural action of $\GL(n, 2)$ on $\F_2^n\wedge \F_2^n$. 

For convenience, we write the group operations in $\Phi(F_{\Phi\text{-}2, 2, n})$ in the additive notation. It is identified with $\F_2^{\binom{n}{2}+n}$ (the linear space of length-$(\binom{n}{2}+n)$ column vectors over $\F_2$), with an ordered linear basis being $(x_1^2, \dots, x_n^2, [x_1, x_2], \dots, [x_{n-1}, x_n])$. Let $P=\linspan\{x_i^2\mid i\in[n]\}$ and $D=\linspan\{[x_i,x_j]\mid 1\leq i<j\leq n\}$. The induced action of $\Aut(F_{\Phi\text{-}2, 2, n}/\Phi(F_{\Phi\text{-}2, 2, n}))\cong\GL(n, 2)$ on $\Phi(F_{\Phi\text{-}2, 2, n})$ can be verified as a linear action, with $D$ being an invariant subspace. That is, let $\tilde A$ be the induced action of $P\in\GL(n, 2)$ on $\Phi(F_{\Phi\text{-}2, 2, n})$. In the ordered basis above, we have \begin{equation}\label{eq:induced-action}
	\tilde A=\begin{bmatrix} B & 0 \\ C & E\end{bmatrix}
\end{equation}
where $E$ is of size $\binom{n}{2}\times\binom{n}{2}$ and representing the natural action of $P$ on $\F_2^n\wedge \F_2^n$. 

By \cite[Lemma 4.3]{BNV07}, $f_{\higp(n, \ell-n)}$ is equal to the number of isomorphism classes of codimension-$(\ell-n)$ subspaces of $\Phi(F_{\Phi\text{-}2, 2, n})$ under this action. Let $\Phi(F_{\Phi\text{-}2, 2, n})^*$ be the dual space of $\Phi(F_{\Phi\text{-}2, 2, n})$, with an ordered linear basis being $((x_1^2)^*, \dots, (x_n^2)^*, [x_1, x_2]^*, \dots, [x_{n-1}, x_n]^*)$. Let $P^*=\linspan\{(x_i^2)^*\mid i\in[n]\}$ and $D=\linspan\{[x_i,x_j]^*\mid 1\leq i<j\leq n\}$.  Naturally, $f_{\higp(n, \ell-n)}$ is also equal to the number of isomorphism classes of dimension-$(\ell-n)$ subspaces of $\Phi(F_{\Phi\text{-}2, 2, n})$ under the dual of this action. 

We now observe the following. 
\begin{fact}\label{fact:dual}
	Let $\vec{v}=\sum_{i\in[n]}a_i(x_i^2)^*+\sum_{1\leq j<k\leq n}b_{j,k}[x_j,x_k]^*\in \Phi(F_{\Phi\text{-}2, 2, n})^*$. If $\vec{v}$ is fixed by the induced action of $P\in \Aut(F_{\Phi\text{-}2, 2, n}/\Phi(F_{\Phi\text{-}2, 2, n}))$, then $\sum_{1\leq j<k\leq n}b_{j,k}[x_j,x_k]^*$ is fixed by the induced action of $P$ on $(\F_2^n\wedge \F_2^n)^*$.
\end{fact}
\begin{proof}
	By Equation~\eqref{eq:induced-action}, the dual action of $P$ on the ordered basis $((x_1^2)^*, \dots, (x_n^2)^*, [x_1, x_2]^*,$ $\dots, [x_{n-1}, x_n]^*)$ is $\tilde A^{-t}=\begin{bmatrix} B' & C' \\ 0 & E'\end{bmatrix}$, where $E'$ is of size $\binom{n}{2}\times\binom{n}{2}$. Because of the $0$ left-lower block, we see that $\sum_{i\in[n]}a_i(x_i^2)^*$ do not affect the $E^*$-component of $\tilde A^{-t}\vec{v}$. Furthermore, $E'=E^{-t}$ records the dual action of $P$ on $V\wedge V$.
\end{proof}

Therefore, define $g(D^*\oplus P^*, m)$ as the number of orbits of the action of $\GL(n, 2)$ on $\Phi(F_{\Phi\text{-}2, 2, n})^*=D^*\oplus P^*$. Let $W\cong \F_2^m$, and let $g((D^*\oplus P^*)\otimes W)$ be the number of orbits of $\GL(n, 2)\times\GL(W)$ on $(D^*\oplus P^*)\otimes W$. When
$m=\lceil C\cdot n+R\rceil$ for constants $C$ and $R$, we claim that
$$g((D^*\oplus P^*)\otimes W)\leq (q-1+\frac{1}{q^{\Omega(n^2)}})\cdot \frac{q^{(\binom{n}{2}+n)\cdot m}}{|\GL(V)|\cdot|\GL(W)|},$$ and 
$$g(D^*\oplus P^*, m)\leq K\cdot (q-1+\frac{1}{q^{\Omega(n^2)}})\cdot \frac{|\Gr(V\wedge V\oplus V, m)|}{|\GL(V)|}.$$
This can be proved in the same way as the proof of Theorem~\ref{thm: tensors_frattini}. The key is to show that the analogue statement in Observation~\ref{obs:frattini} holds, thanks to Fact~\ref{fact:dual}. This concludes the proof for the $\higtwo$ case.

\medskip
\paragraph{Acknowledgements.} Y. L.'s research is supported in part by the National Key Research and Development Program of China (No. 2024YFE0102500), the National Nature Science Foundation of China (No. 62302346, No. 12441101), the Hubei Provincial Natural Science Foundation of China (No. 2024AFA045) and the “Fundamental Research Funds for the Central Universities”. 

Youming Qiao's research is supported in part by Australian Research Council DP200100950 and LP220100332. Part of this work was done while Youming was a member of the Institute for Advanced Study in Princeton supported by the Ky Fan and Yu-Fen Fan Endowment Fund.

\bibliographystyle{alpha}
\bibliography{references}

\end{document}